\documentclass[12pt]{article}
\usepackage{graphicx}
\usepackage{amsmath}
\usepackage{amsfonts}
\usepackage{amsthm}
\usepackage[T1]{fontenc}
\usepackage{url}
\usepackage[margin=1in]{geometry}
\newcommand{\ud}{\,\mathrm{d}}

\begin{document}
\title{New asymptotically sharp Korn and Korn-like inequalities in thin domains}
\author{Davit Harutyunyan\\
\textit{Department of Mathematics, The University of Utah}}

\maketitle
\begin{abstract}
It is well known that Korn inequality plays a central role in the theory of linear elasticity. In the present work
we prove new asymptotically sharp Korn and Korn-like inequalities in thin curved domains with a non-constant thickness.
This new results will be useful when studying the buckling of compressed shells, in particular when calculating the critical
buckling load.
\newline
\linebreak
\end{abstract}

\tableofcontents

\section{Introduction and the main results}

\newtheorem{Theorem}{Theorem}[section]
\newtheorem{Lemma}[Theorem]{Lemma}
\newtheorem{Corollary}[Theorem]{Corollary}
\newtheorem{Remark}[Theorem]{Remark}
\newtheorem{Definition}{Definition}[section]

 Korn and Korn-like Inequalities are essential for proving the existence of a solution to the main boundary value problems of elasticity and for estimating
the solutions, see [\ref{bib.Friedrichs},\ref{bib.Horgan},\ref{bib.Kondratiev.Oleinik2},\ref{bib.Korn1},\ref{bib.Korn2}]. Assume $\Omega\in\mathbb R^n$ is a bounded and connected domain with a Lipschitz boundary. For any
displacement $U=(u_1,u_2,\dots,u_n)$ denote by $e(U)$ the symmetric part of the gradient, i.e., $e(U)=\frac{1}{2}(\nabla U+\nabla U^t),$ which will be the strain in the elasticity context. Then the second Korn inequality asserts the following: \textit{There exists a constant $C(\Omega)$ depending only on $\Omega$ such that}
\begin{equation}
\label{Korn.ineq.2}
\|\nabla U\|_{L^2(\Omega)}^2\leq C(\Omega)(\|e(U)\|_{L^2(\Omega)}^2+\|U\|_{L^2(\Omega)}^2)\quad\textit{for all}\quad U\in H^1(\Omega).
\end{equation}
Denote by $skew(\mathbb R^n)$ the space of rigid displacements of $\mathbb R^n,$ i.e., the set of all vector fields $U=(u_1,u_2,\dots,u_n)$ such that $U(x)=a+Ax,$ where $a\in\mathbb R^n$ and $A$ is an $n\times n$ skew-symmetric matrix. Assume furthermore, that $V$ is a closed subspace of $H^1(\mathbb R^n,\mathbb R^n)$ that has no intersection with $skew(\mathbb R^n),$ except the identically zero transformation. Then the first Korn inequality asserts the following:
\textit{There exists a constant $C(\Omega,V)$ depending only on $\Omega$ and the subspace $V$ such that}
\begin{equation}
\label{Korn.ineq.1}
\|\nabla U\|_{L^2(\Omega)}^2\leq C(\Omega,V)\|e(U)\|_{L^2(\Omega)}^2,\quad \textit{for all} \quad U\in V.
\end{equation}
Another version of first Korn inequality, called geometric rigidity also holds and asserts the following:
\textit{There exists a constant $C(\Omega)$ depending only on $\Omega$ such that for any displacement  $U\in H^1(\Omega)$ there exists an associated skew-symmetric matrix $A_U\in skew(\mathbb R^n)$ such that}

\begin{equation}
\label{Korn.ineq.1.geomertic}
\|\nabla U-A_U\|_{L^2(\Omega)}^2\leq C(\Omega)\|e(U)\|_{L^2(\Omega)}^2.
\end{equation}

In particular, geometric rigidity estimate (\ref{Korn.ineq.1.geomertic}) implies that if a displacement $U\in H^1(\Omega)$ has a skew-symmetric gradient a.e. in $\Omega,$ then it must be a constant affine transformation with a skew-symmetric gradient.
There have been several proofs of first and second Korn inequalities since Korn's work in 1909, see [\ref{bib.Korn2}]. In [\ref{bib.Kondratiev.Oleinik1},\ref{bib.Kondratiev.Oleinik2}] Kondratiev and Oleinik gave very elegant proofs to different kind of Korn and Korn-like inequalities. However, there have been only a few results that give
the exact asymptotics of the constant $C$ in the first or second Korn inequalities for different domains with different vector spaces $V,$ see [\ref{bib.Grab.Har1},\ref{bib.Grab.Har2},\ref{bib.Lewiz.Muller},\ref{bib.Lewicka.Muller1}] for some recent results in thin domains, that concern the exact asymptotics of the constant $C$ in terms of the thickness of the domain. It has been understood in [\ref{bib.Grab.Har2}] that such kind of Korn inequalities play a crucial role in
establishing the critical buckling load in shell buckling problems, in particular a rigorous asymptotic analysis of the buckling of perfect cylindrical shells under axial compression is done in [\ref{bib.Grab.Har2}].
Another issue is the so called "sensitivity to imperfections" in the mentioned problem, which corresponds to the case of an imperfect load or imperfect cylindrical shell. In the case of shape imperfections one needs new sharp Korn inequalities for the cross sections of the cylinder i.e., Korn inequalities for two dimensional domains, to prove an asymptotically sharp Korn inequality for the shell, and therefore to attempt to extend the results in [\ref{bib.Grab.Har2}] to the generalized cylindrical shells. In the present work we prove asymptotically sharp Korn and Korn type inequalities in thin curved domains with a nonconstant thickness in all dimensions and also in dimension two with zero and periodic Dirichlet type boundary conditions in the thickness direction. The new inequalities in $2D$ can be used to prove a Korn inequality for three dimensional shells with a nonconstant thickness, in particular for cylindrical shells, spherical shells, etc. We emphasize, that the inequalities in the present work generalize some of the ones obtained in [\ref{bib.Grab.Har1}] that are the cornerstones of the analysis in [\ref{bib.Grab.Har1},\ref{bib.Grab.Har2}] and also that they are not straightforward extensions of the ones in [\ref{bib.Grab.Har1}] as the proofs are based on totally different ideas. Lemma~\ref{lemma1.ellipcit.korn} extends a similar lemma in [\ref{bib.Grab.Har1}] to any elliptic operator with constant coefficients instead of the Laplace as well as $\mathbb R^2$ to $\mathbb R^n.$ The operator extension is essential, because when making a linear change of variables like $y_1=\sum_{i=1}^na_ix_i,$ $y_i=x_i,$ when $i\geq 2,$ to rotate the domain, the Laplace operator becomes an elliptic one with constant coefficients. For an operator $L(u)=\sum_{i,j=1}^na_{ij}\frac{\partial^2u}{\partial x_i \partial x_j}$ with constant coefficients we will consider the elipticity condition
\begin{equation}
\label{ellipticity}
\sum_{i,j=1}^na_{ij}x_ix_j\geq\lambda |x|^2\quad \text{for all} \quad x\in \mathbb R^n,
\end{equation}
where $\lambda>0,$ and we also assume that
\begin{equation}
\label{bounded.coeff}
\sum_{i=1}^n|a_{ij}|\leq \Lambda\quad\text{for all}\quad 1\leq j\leq n.
\end{equation}
We call in this paper an inequality like the one given in Theorem~\ref{theorem1.second.korn.dim2} a \textit{strong second Korn inequality}, as it implies the usual second Korn inequality. We will as well call an inequality like the one given in Theorem~\ref{theorem2.ellipcit.korn} a \textit{Korn-like inequality}, because it actually derives a two dimensional strong second Korn inequality, as will be seen later.
The following two theorems are \textit{Korn-like inequalities} in all dimensions, where for the vector $x=(x_1,x_2,\dots,x_n)$ we set $x'=(x_2,x_3,\dots,x_n).$

\begin{Theorem}
\label{theorem2.ellipcit.korn}
Let $\omega\subset\mathbb R^{n-1}$ be a bounded Lipschitz domain, let $x_1=\varphi(x')\colon\omega\to\mathbb R$ be a positive Lipschitz function with $H=\sup_{x'\in\omega}\varphi(x')$ and $h=\inf_{x'\in\omega}\varphi(x')>0.$ Denote $\Omega=\{x\in\mathbb R^n \ : \ x'\in\omega, \ 0<x_1<\varphi(x')\}$ and assume that the operator $L(u)=\sum_{i,j=1}^na_{ij}\frac{\partial^2u}{\partial x_i \partial x_j}$ with constant coefficients satisfies conditions (\ref{ellipticity}) and (\ref{bounded.coeff}). Then there exists a constant $C$ depending on $n,$ $\Lambda,$ $\lambda,$ $L=\mathrm{Lip}(\varphi)$ and the ratio $m=H/h$ such that any $u\in C^3(\bar\Omega)$ solution of $L(u)=0$ satisfying the boundary conditions $u(x)=0$ on the portion $\Gamma=\{x\in\partial\Omega \ : \ x'\in\partial\omega\}$ of the boundary of $\Omega$
fulfills the inequality
$$\|\nabla u\|_{L^2(\Omega)}^2\leq \frac{C(n,\lambda,\Lambda,L,m)}{h}\|u\|_{L^2(\Omega)}\|u_{x_1}\|_{L^2(\Omega)}+C(n,\lambda,\Lambda, L,m)\|u_{x_1}\|_{L^2(\Omega)}^2.$$
\end{Theorem}

\begin{Remark}
The estimate in Theorem~\ref{theorem2.ellipcit.korn} is sharp in the sense, that in the case when $\Omega=[0,h]\times[a,b]$ and $L$ is the Laplace operator, then the inequality becomes an equality for the function
$$u(x)=\cosh\left(\frac{\pi}{b-a}\left(x_1-\frac{h}{2}\right)\right)\sin\left(\frac{\pi x_2}{b-a}\right),$$
provided by Grabovsky and Harutyunyan in [\ref{bib.Grab.Har1}].
\end{Remark}

The next theorem is a variant of Theorem~\ref{theorem2.ellipcit.korn} with the left boundary of the domain being a part of a hyperplane not necessarily perpendicular to the $x_1$ axis.
\begin{Theorem}
\label{theorem3.ellipcit.hyperplane}
Let $\omega\subset\mathbb R^{n-1}$ be a bounded Lipschitz domain, let $\alpha$ be the hyperplane $\{x\in\mathbb R^n \ : \ x_1=a_1+\sum_{i=2}^na_ix_i\}$ and let $x_1=\varphi(x')$ be a Lipschitz function with $H=\sup_{x'\in\omega}\big(\varphi(x')-a_1-\sum_{i=2}^na_ix_i\big)$ and $h=\inf_{x'\in\omega}\big(\varphi(x')-a_1-\sum_{i=2}^na_ix_i\big)>0.$ Denote $\Omega=\{x\in\mathbb R^n \ : \ x'\in\omega, \ a_1+\sum_{i=2}^na_ix_i< x_1< \varphi(x')\}.$ Assume that the operator $L(u)=\sum_{i=1}^nb_{i}u_{ii}$ with constant coefficients with no mixed derivatives satisfies conditions (\ref{ellipticity}) and (\ref{bounded.coeff}). Then there exists a constant $C$ depending on $n,$ $\lambda,$ $\Lambda,$ $L=\mathrm{Lip}(\varphi),$ the ratio $m=H/h$ and the number $A=\max_{1\leq i\leq n}|a_i|,$ such that any $u\in C^3(\bar\Omega)$ solution of $L(u)=0$ in $\Omega$ satisfying the boundary conditions $u(x)=0$ on the portion $\Gamma=\{x\in\partial\Omega \ : \ x'\in\partial\omega\}$ of the boundary of $\Omega$
fulfills the inequality
$$\|\nabla u\|_{L^2(\Omega)}^2\leq \frac{C(n,\lambda,\Lambda,L,m,A)}{h}\|u\|_{L^2(\Omega)}\|u_{x_1}\|_{L^2(\Omega)}+C(n,\lambda,\Lambda,L,m,A)\|u_{x_1}\|_{L^2(\Omega)}^2.$$
\end{Theorem}

The following theorem is a \textit{strong second Korn inequality} in two dimensional thin curved domains with a nonconstant thickness, which corresponds to the cross sections taken in the angular direction of possibly a cylindrical or spherical shell.
\begin{Theorem}
\label{theorem1.second.korn.dim2}
 Let $l>0,$ let $\varphi_1\in C^1[0,l]$ and let $\varphi_2$ and $\varphi_1'$ be Lipschitz functions defined on $[0,l].$ Assume furthermore that $0<h=\min_{y\in [0,l]}(\varphi_2(y)-\varphi_1(y))$ and $H=\min_{y\in[0,l]}(\varphi_2(y)-\varphi_1(y)).$ Denote $\Omega=\{(x,y)\in\mathbb R^2 \ : \ y\in(0,l),\  \varphi_1(y)< x < \varphi_2(y)\}.$ Then there exists a constant $C$ depending on $m=H/h,$ $\rho_1=\|\varphi_{1}'\|_{L^\infty(\Omega)},$ $\rho_2=\|\varphi_{2}'\|_{L^\infty(\Omega)}$ and $\rho_1'=\|\varphi_{1}''\|_{L^\infty(\Omega)}$ such that if the first component of the displacement $U=(u,v)\in W^{1,2}(\Omega)$ satisfies the boundary conditions $u(x)=0$ on the boundary portion $\Gamma=\{(x,y)\in\partial\Omega \ : \ y=0\ \text{or}\ y=l\}$ in the trace sense, then the strong second Korn inequality holds:

$$\|\nabla U\|_{L^2(\Omega)}^2\leq \frac{C(m,\rho_1,\rho_2,\rho_1')}{h}\|u\|_{L^2(\Omega)}\|e(U)\|_{L^2(\Omega)}+C(m,\rho_1,\rho_2,\rho_1')\|e(U)\|_{L^2(\Omega)}^2.$$
\end{Theorem}

\begin{Remark}
The estimate in Theorem~\ref{theorem1.second.korn.dim2} is sharp in the sense, that in the case when $\varphi_1\equiv0,$ it becomes an equality for the displacement
$$U=\left(f\left(\frac{y}{h^\alpha}\right),-\frac{x}{h^\alpha}f\left(\frac{y}{h^\alpha}\right)\right),$$
where $f$ is a smooth function supported on $[0,l]$ and $\alpha\in[0,1/2],$ $h\in(0,1).$
\end{Remark}

\begin{Corollary}
\label{1stKorninequality}
Under the conditions of Theorem~\ref{theorem1.second.korn.dim2} first Korn inequality holds:
\begin{equation}
\label{5}
\|\nabla U\|_{L^2(\Omega)}^2\leq \frac{C}{h^2}\|e(U)\|_{L^2(\Omega)}^2.
\end{equation}
\end{Corollary}

\begin{proof}
By the Friedrich inequality we have
\begin{equation}
\label{4fr}
\|u\|_{L^2(\Omega)}^2\leq C_{Fr}\|\nabla U\|_{L^2(\Omega)}^2,
\end{equation}
 where $C_{Fr}$ depends on the Lipschitz character of $\Omega.$ Theorem~\ref{theorem1.second.korn.dim2} and (\ref{4fr}) together complete the proof.
\end{proof}

\begin{Remark}
Inequality (\ref{5}) is a generalization of first Korn inequality on rectangles proven by Ryzhak in [\ref{bib.Ryzhak}] and studied later by Grabovsky and Truskinovsky in [\ref{bib.Grab.Trus}].
\end{Remark}

The last theorem is a \textit{strong second Korn inequality} in two dimensional thin curved domains with a nonconstant thickness, which corresponds to cross sections in the thickness direction of possibly a cylindrical or spherical shell.
\begin{Theorem}
\label{theorem.dim2.pariodic}
  Let $l>0,$ $\varphi_1,$ $\varphi_2,$ $\Omega,$ $h,$ $H,$ $m,$ $\rho_1,$ $\rho_2$ and $\rho_1'$ be as in Theorem~\ref{theorem1.second.korn.dim2}. Then there exists a constant $C$ depending on $m,$ $\rho_1,$ $\rho_2$ and $\rho_1'$ such that if the first component $u$ of the displacement $U=(u,v)\in W^{1,2}(\Omega)$ is $l$-periodic, then the second Korn inequality holds:

$$\|\nabla U\|_{L^2(\Omega)}^2\leq \frac{C(m,\rho_1,\rho_2,\rho_1')}{h}\|u\|_{L^2(\Omega)}\|e(U)\|_{L^2(\Omega)}+C(m,\rho_1,\rho_2,\rho_1')\|e(U)\|_{L^2(\Omega)}^2.$$
\end{Theorem}

\section{Preliminary}
\label{section.prem}

We start with a lemma that extends the Laplace operator to any elliptic operator with constant coefficients in Lemma 2.2 in [\ref{bib.Oleinik.shamaev.Yosifian}].

\begin{Lemma}
\label{lemma,elliptic.gradient}
Let $\Omega$ be a bounded open domain in $\mathbb R^n$ with a Lipschitz boundary, and let $\partial\Omega=\Gamma_1\cup\Gamma_2.$ Denote by $\delta$ the distance function from the $\Gamma_1$ part of the boundary of  $\Omega,$ i.e., $\delta(x)=\mathrm{dist}(x,\Gamma_1)$ for $x\in\mathbb R^n.$ Let
the operator $L(u)=\sum_{i,j=1}^na_{ij}\frac{\partial^2u}{\partial x_i\partial x_j}$ with constant coefficients satisfy conditions (\ref{ellipticity}) and ({\ref{bounded.coeff}}). Assume furthermore that the function $f\in C^2(\bar \Omega,\mathbb R)$ satisfies the boundary condition $f(x)=0$ on $\Gamma_2.$ Then
\begin{equation}
\|\delta\nabla f\|_{L^2(\Omega)}^2\leq \Big(\frac{4n\Lambda^2}{\lambda^2}+1\Big)\|f\|_{L^2(\Omega)}^2
+\frac{1}{\lambda^2}\|\delta^2L(f)\|_{L^2(\Omega)}^2.
\end{equation}
\end{Lemma}

\begin{proof}

Utilizing elipticity condition (\ref{ellipticity}) we get integrating by parts,
\begin{align*}
\lambda\int_{\Omega}\delta^2|\nabla f|^2\ud x&\leq \int_{\Omega}\delta^2\sum_{i,j=1}^na_{ij}\frac{\partial f}{\partial x_i}\frac{\partial f}{\partial x_j}\\
&=-\int_{\Omega}f\cdot \sum_{i=1}^n \frac{\partial}{\partial x_i}\Big(\delta^2\sum_{j=1}^na_{ij}\frac{\partial f}{\partial x_j}\Big)\ud x\\
&=-\int_{\Omega}f\cdot\delta^2L(f)\ud x-2\int_{\Omega}f\cdot\sum_{i,j=1}^n a_{ij}\delta\frac{\partial\delta }{\partial x_i}\frac{\partial f }{\partial x_j}\ud x.
\end{align*}
Notice that $\delta$ is a Lipschitz function with partial derivatives bounded by one, thus we have by the Schwartz inequality,

\begin{align*}
2\Big|\int_{\Omega}f\cdot\sum_{i,j=1}^n a_{ij}\delta\frac{\partial\delta }{\partial x_i}\frac{\partial f }{\partial x_j}\ud x\Big|&\leq
2\int_{\Omega}\Lambda|f|\delta\cdot\sum_{i=1}^n\Big|\frac{\partial f }{\partial x_i}\Big|\ud x\\
&\leq \frac{2\Lambda^2n}{\lambda}\int_{\Omega}|f|^2\ud x+\frac{\lambda}{2n}\int_{\Omega}\delta^2\Big(\sum_{i=1}^n\big|\frac{\partial f }{\partial x_i}\big|\Big)^2\ud x\\
&\leq \frac{2\Lambda^2n}{\lambda}\int_{\Omega}|f|^2\ud x+\frac{\lambda}{2}\int_{\Omega}|\delta\nabla f|^2\ud x.
\end{align*}

Applying again the Schwartz inequality we can estimate the first summand as follows

$$\Big|\int_{\Omega}f\cdot\delta^2L(f)\ud x\Big|\leq \frac{\lambda}{2}\int_{\Omega}|f|^2\ud x+\frac{1}{2\lambda}\int_{\Omega}|\delta^2L(f)|^2\ud x.$$
Combining now the obtained estimates we discover,
$$\|\delta\nabla f\|_{L^2(\Omega)}^2\leq \Big(\frac{2\Lambda^2n}{\lambda^2}+\frac{1}{2}\Big)\|f\|_{L^2(\Omega)}^2
+\frac{1}{2\lambda^2}\|\delta^2L(f)\|_{L^2(\Omega)}^2+\frac{1}{2}\|\delta\nabla f\|_{L^2(\Omega)}^2,$$
which completes the proof.
\end{proof}

The next lemma is a variant of Lemma~\ref{lemma,elliptic.gradient} in the space dimension two, with some nonconstant coefficient operator. As will be seen later in Section~\ref{section.fixed.boundary.conditions}, such kind of operator appears when mapping a curved region onto another one with a left boundary being a vertical segment, see proof of Theorem~\ref{theorem1.second.korn.dim2}.

\begin{Lemma}
\label{lemma,elliptic.gradient2}
Let $\Omega,$ $\Gamma_1,$ $\Gamma_2$ and $\delta$ be as in Lemma~\ref{lemma,elliptic.gradient} with $n=2.$ Assume $a(y)\in C^1(\bar\Omega,\mathbb R)$ and consider the operator $L_a(u)=(1+a^2(y))u_{xx}-2a(y)u_{xy}+u_{yy}-a'(y)u_x.$ Then there exists a constant $C$ depending on the quantities $M=\|a\|_{L^\infty(\bar\Omega)}$ and $M_1=\|a'\|_{L^\infty(\bar\Omega)},$ such that
if the function $f\in C^2(\bar\Omega,\mathbb R)$ satisfies the condition
\begin{equation}
\label{Zero.boundary.integral}
\int_{\partial\Omega}f\delta^2(f_x\nu_1(1+a^2(y))-2a(y)f_x\nu_2+f_y\nu_2)\ud S=0,
\end{equation}
where $\nu=(\nu_1,\nu_2)$ is the outward unit normal to $\partial\Omega,$ then there holds
$$
\|\delta\nabla f\|_{L^2(\Omega)}^2\leq C(M,M_1)\big(\|(1+\delta)f\|_{L^2(\Omega)}^2+\|\delta^2L_a(f)\|_{L^2(\Omega)}^2\big).
$$

\end{Lemma}

\begin{proof}

It is straightforward that
$$(1+a^2)t^2-2ats+s^2\geq \lambda_a(t^2+s^2)\quad\text{for any}\quad (s,t)\in\mathbb R^2,$$
where
\begin{equation}
\label{l_a>l}
\lambda_a=\frac{2}{2+a^2+a\sqrt{4+a^2}}\geq\frac{2}{2+M^2+M\sqrt{4+M^2}}:=\lambda.
\end{equation}
Like in the proof of Lemma~\ref{lemma,elliptic.gradient}, we will have zero boundary term integral when doing integration by parts as condition (\ref{Zero.boundary.integral}) ensures, thus we obtain,

\begin{align*}
\lambda_a\int_{\Omega}|\delta\nabla f|^2 &\leq \int_{\Omega}\delta^2\big((1+a^2)f_{x}^2-2af_xf_y+f_y^2\big)\\
&=-\int_{\Omega}f\delta^2L_a(f)-2\int_{\Omega}f\delta\big((1+a^2)\delta_xf_x-2a\delta_xf_y+\delta_yf_y\big)-\int_{\Omega}ff_x\delta^2a'.
\end{align*}
We estimate the summands on the left hand side of the above inequality by the Schwartz inequality as follows:
\begin{align*}
-\int_{\Omega}f\delta^2L_a(f)\leq \frac{1}{2}\int_{\Omega}|f|^2+\frac{1}{2}\int_{\Omega}|\delta^2L_a(f)|^2,
\end{align*}
\begin{align*}
-2\int_{\Omega}f\delta\big((1+a^2)\delta_xf_x&\leq 2(1+M^2)\int_{\Omega}|f\delta f_x|\\
&\leq \frac{8(1+M^2)^2}{\lambda_a}\int_{\Omega}|f|^2+\frac{\lambda_a}{8}\int_{\Omega}|\delta \nabla f|^2,
\end{align*}
\begin{align*}
4\int_{\Omega}f\delta a\delta_xf_y&\leq 4M\int_{\Omega}|f\delta f_y|\\
&\leq \frac{32M^2}{\lambda_a}\int_{\Omega}|f|^2+\frac{\lambda_a}{8}\int_{\Omega}|\delta \nabla f|^2,
\end{align*}
\begin{align*}
-2\int_{\Omega}f\delta\delta_yf_y&\leq 2\int_{\Omega}|f\delta f_y|\\
&\leq  \frac{8}{\lambda_a}\int_{\Omega}|f|^2+\frac{\lambda_a}{8}\int_{\Omega}|\delta \nabla f|^2,
\end{align*}
and
\begin{align*}
-\int_{\Omega}ff_x\delta^2a'&\leq M_1\int_{\Omega}|ff_x\delta^2|\\
&\leq \frac{2M_1^2}{\lambda_a}\int_{\Omega}|\delta f|^2+\frac{\lambda_a}{8}\int_{\Omega}|\delta \nabla f|^2.
\end{align*}

Combining the obtained inequalities we obtain

$$\lambda_a\int_{\Omega}|\delta\nabla f|^2\leq \frac{\lambda_a}{2}\int_{\Omega}|\delta \nabla f|^2+\frac{M_2}{\lambda_a}\int_{\Omega}|(1+\delta)f|^2+\frac{1}{2}\int_{\Omega}|\delta^2L_a(f)|^2,$$
where $M_2$ depends only on $M$ and $M_1.$ Finally dividing the inequality by $\lambda_a$ and utilizing the inequality $\lambda_a\geq \lambda$ we
finish the proof.
\end{proof}

\begin{Remark}
\label{Zero.imply.zero.bd.integral}
Notice that if $f(x)=0$ for $x\in\Gamma_2$ then condition (\ref{Zero.boundary.integral}) is fulfilled.
\end{Remark}
\begin{proof}
It is evident that $f\delta=0$ on $\partial \Omega,$ thus the proof follows.
\end{proof}
The following lemma is a variant of Hardy inequality, a partial case of which is proven in [\ref{bib.Oleinik.shamaev.Yosifian}, Theorem 2.1].
It actually allows to extend Korn type inequalities held in the cylindrical domains with axis parallel to $x_1$ direction,
to domains with the left boundary, orientation taken in the $x_1$ direction, being a part of a hyperplane and the right boundary, orientation taken again in the $x_1$ direction, being a Lipschitz surface. Basically, it derives Theorem~\ref{theorem2.ellipcit.korn} from Lemma~\ref{lemma1.ellipcit.korn} as will be seen in the next section.
\begin{Lemma}
\label{lemma,a,epsilon}
Let $b>a>0,$ $\epsilon\in(0,1],$ and let $f\in C^1[a,b].$ Then
$$\int_{a+\epsilon(b-a)}^bf^2(t)\ud t\leq \frac{2}{\epsilon}\int_a^{a+\epsilon(b-a)}f^2(t)\ud t+4\int_a^{b}f'^2(t)(b-t)^2\ud t.$$
\end{Lemma}

\begin{proof}
First of all notice that the inequality is invariant under variable translation, therefore one can without loss of generality
assume that $a=0.$ By the Schwartz inequality and by integration by parts we have for any $x\in ((1-\epsilon)b,b)$

\begin{align*}
\int_0^{x}f^2(t)\ud t&=xf^2(x)-2\int_0^{x}tf'(t)f(t)\ud t\\
&\leq xf^2(x)+\frac{1}{2}\int_0^{x}f^2(t)\ud t+2\int_0^{x}t^2f'^2(t)\ud t\\
&\leq xf^2(x)+\frac{1}{2}\int_0^{x}f^2(t)\ud t+2\int_0^bt^2f'^2(t)\ud t,
\end{align*}

thus
\begin{equation}
\label{1}
\int_0^{x}f^2(t)\ud t\leq 2xf^2(x)+4\int_0^bt^2f'^2(t)\ud t.
\end{equation}
By the mean value formula we can choose the point $x\in ((1-\epsilon)b,b)$ such that
$$f^2(x)=\frac{1}{\epsilon b}\int_{(1-\epsilon)b}^bf^2(t)\ud t,$$
thus with this choice of $x$ estimate (\ref{1}) will imply
\begin{align*}
\int_0^{(1-\epsilon)b}f^2(t)\ud t&\leq \int_0^{x}f^2(t)\ud t\\
&\leq 2xf^2(x)+4\int_0^bt^2f'^2(t)\ud t\\
&\leq \frac{2x}{\epsilon b}\int_{(1-\epsilon)b}^bf^2(t)\ud t+4\int_0^bt^2f'^2(t)\ud t\\
&\leq \frac{2}{\epsilon}\int_{(1-\epsilon)b}^bf^2(t)\ud t+4\int_0^bt^2f'^2(t)\ud t.
\end{align*}

Now, an application of the last inequality to the function $g(t)=f(b-t)$ with the variable change $t=b-x$
completes the proof.
\end{proof}

\section{Korn like inequalities in all dimensions}
\label{section.Boundsonthegradient}

Next we prove a \textit{Korn-like} inequality for the solutions of elliptic PDEs in cylindrical domains.
\begin{Lemma}
\label{lemma1.ellipcit.korn}
Let $\omega\subset\mathbb R^{n-1}$ be a bounded Lipschitz domain, let $h>0$ and $\Omega=[0,h]\times\omega.$ Assume the operator $L(u)=\sum_{i,j=1}^na_{ij}\frac{\partial^2u}{\partial x_i \partial x_j}$ with constant coefficients satisfies conditions (\ref{ellipticity}) and (\ref{bounded.coeff}). Then there exists a constant $C$ depending on $n,$ $\lambda$ and $\Lambda$ such that any $u\in C^3(\bar\Omega)$ solution of
$L(u)=0$ satisfying the boundary conditions $u(x)=0$ on $[0,h]\times\partial\omega$
fulfills the inequality
$$\|\nabla u\|_{L^2(\Omega)}^2\leq \frac{C(n,\lambda,\Lambda)}{h}\|u\|_{L^2(\Omega)}\|u_{x_1}\|_{L^2(\Omega)}+C(n,\lambda,\Lambda)\|u_{x_1}\|_{L^2(\Omega)}^2.$$
\end{Lemma}

\begin{proof}
For any $t\in[0,h/2]$ denote $\Omega_t=[h/2-t,h/2+t]\times\omega$ and $\Omega_t'=[0,t]\times\omega.$ We have integrating by parts,
\begin{align*}
\lambda\int_{\Omega_t}|\nabla u|^2\ud x&\leq\int_{\Omega_t}\sum_{i,j=1}^na_{ij}\frac{\partial u}{\partial x_i}\frac{\partial u}{\partial x_j}\ud x\\
&=a_{11}\left(\int_{\{h/2+t\}\times\omega}u(x')u_{x_1}(x')\ud x'-\int_{\{h/2-t\}\times\omega}u(x')u_{x_1}(x')\ud x'\right)-\int_{\Omega_t}uL(u)\ud x\\
&=a_{11}\left(\int_{\{h/2+t\}\times\omega}u(x')u_{x_1}(x')\ud x'-\int_{\{h/2-t\}\times\omega}u(x')u_{x_1}(x')\ud x'\right),
\end{align*}
 where $x'=(x_2, x_3,\dots,x_n).$ Thus we can estimate
$$\lambda\int_{\Omega_t}|\nabla u|^2\ud x\leq \Lambda\int_{\{h/2+t\}\times\omega}|u(x')u_{x_1}(x')|\ud x'+\Lambda\int_{\{h/2-t\}\times\omega}|u(x')u_{x_1}(x')|\ud x'.$$
Integrating now the last inequality over $[0,h/2]$ and utilizing the Schwartz inequality we discover

$$\lambda\int_{0}^{h/2}\ud t\int_{\Omega_t}|\nabla u|^2\ud x\leq \Lambda\|u\|_{L^2(\Omega)}\|u_{x_1}\|_{L^2(\Omega)}.$$

Notice that the function $\rho(t)=\int_{\Omega_t}|\nabla u|^2$ is nonnegative and increasing in $[0,h/2]$, therefore the last inequality implies
\begin{align*}
\frac{1}{\lambda\Lambda}\|u\|_{L^2(\Omega)}\|u_{x_1}\|_{L^2(\Omega)}&\geq \int_{0}^{h/2}\ud t\int_{\Omega_t}|\nabla u|^2\ud x\\
&=\int_{0}^{h/2}\rho(t)\ud t\\
&\geq\int_{h/4}^{h/2}\rho(t)\ud t\\
&\geq\int_{h/4}^{h/2}\rho(h/4)\ud t\\
&=\frac{h}{4}\int_{\Omega_{h/4}}|\nabla u|^2\ud x,
\end{align*}

thus we obtain
\begin{equation}
\label{2}
\frac{h}{4}\int_{\Omega_{\frac{h}{4}}}|\nabla u|^2\ud x
\leq\frac{1}{\lambda\Lambda}\|u\|_{L^2(\Omega)}\|u_{x_1}\|_{L^2(\Omega)}
\end{equation}
Next we fix any index $2\leq i\leq n,$ point $x'\in\omega$ and apply Lemma~\ref{lemma,a,epsilon} to the function $u_{x_i}$ on the segment with endpoints $(0,x')$ and $(\frac{h}{2},x')$. We have that for $\epsilon=\frac{1}{2},$

$$\int_{0}^\frac{h}{4}u_{x_i}^2(x_1,x')\ud x_1\leq 4\int_{\frac{h}{4}}^\frac{h}{2}u_{x_i}^2(x_1,x')\ud x_1+4\int_{0}^\frac{h}{2}u_{x_1x_i}^2(x_1,x')x_1^2\ud x_1,$$
which integrating over $\omega$ and summing up in $i$ we obtain
\begin{equation}
\label{2.1}
\int_{\Omega_\frac{h}{4}'}|\nabla u|^2\ud x\leq 4\int_{\Omega_\frac{h}{2}'\setminus\Omega_\frac{h}{4}'}|\nabla u|^2\ud x+
\int_{\Omega_\frac{h}{4}'}u_{x_1}^2(x)\ud x+4\int_{\Omega_\frac{h}{2}'}|\nabla u_{x_1}|^2x_1^2\ud x.
\end{equation}

It is clear that $u_{x_1}=0$ on $[0,h]\times\partial\Omega,$ thus we can apply Lemma~\ref{lemma,elliptic.gradient}
to the function $u_{x_1}$ in the domain $\Omega,$ therefore,

$$\int_{\Omega_\frac{h}{2}'}|\nabla u_{x_1}|^2x_1^2\ud x\leq \int_{\Omega}|\delta\nabla u_{x_1}|^2\ud x\leq
\Big(\frac{4\Lambda^2n}{\lambda^2}+1\Big)\int_{\Omega}|u_{x_1}|^2\ud x
+\frac{1}{\lambda^2}\int_{\Omega}|\delta^2L(u_{x_1})|^2\ud x.$$

Differentiating the equality $L(u)=0$ in $x_1$ we get $0=(L(u))_{x_1}=L(u_{x_1}),$ thus the last inequality implies,
\begin{equation}
\label{2.2}
\int_{\Omega_\frac{h}{2}'}|\nabla u_{x_1}|^2x_1^2\ud x\leq\Big(\frac{4\Lambda^2n}{\lambda^2}+1\Big)\int_{\Omega}|u_{x_1}|^2\ud x.
\end{equation}
Coupling now (\ref{2.1}) and (\ref{2.2}) we obtain,
\begin{equation}
\label{2.3}
\int_{\Omega_\frac{h}{4}'}|\nabla u|^2\ud x\leq 4\int_{\Omega_\frac{h}{2}'\setminus\Omega_\frac{h}{4}'}|\nabla u|^2\ud x+
\int_{\Omega_\frac{h}{4}'}u_{x_1}^2(x)\ud x+4\Big(\frac{4\Lambda^2n}{\lambda^2}+1\Big)\int_{\Omega}|u_{x_1}|^2\ud x.
 \end{equation}

It is clear that we can obtain a similar bound on the norm of $\nabla u$ in the right half of $\Omega,$ namely in $\Omega\setminus(\Omega_\frac{h}{4}'\cup\Omega_\frac{h}{4}).$ Taking into account this fact the proof follows now from (\ref{2}) and (\ref{2.3}).

\end{proof}

\textbf{Proof of Theorem~\ref{theorem2.ellipcit.korn}}. The proof is based on an application of Lemmas \ref{lemma,a,epsilon} and \ref{lemma1.ellipcit.korn}.
We apply Lemma~\ref{lemma,a,epsilon} to each of the functions $u_{x_i}$ on the interval $[h/2,\varphi(x')]$ with $\epsilon=\frac{h}{2\varphi(x')-h}.$
Thus we have for any $x'\in\omega$ that
\begin{align*}
\int_{h}^{\varphi(x')}|\nabla u|^2\ud x_1&=\int_{h}^{\varphi(x')}\sum_{i=1}^n|u_{x_i}|^2\ud x_1\\
&\leq \sum_{i=1}^n\Big((4m-2)\int_{\frac{h}{2}}^h|u_{x_i}|^2\ud x_1+4\int_{\frac{h}{2}}^{\varphi(x')}|(x_1-\varphi(x'))u_{x_ix_1}|^2\ud x_1\Big).
 \end{align*}

Since $\varphi(x')$ is Lipschitz, there exists a constant $C_1$ depending on $m$ and the Lipschitz constant of $\varphi(x')$ such that
$|x_1-\varphi(x')|\leq C_1\delta(x)$ uniformly in $x'\in\omega$ and $x_1\in[\frac{h}{2}, \varphi(x')].$ Therefore the last inequality implies
\begin{align*}
\int_{h}^{\varphi(x')}|\nabla u|^2\ud x_1&\leq \sum_{i=1}^n\Big(4m-2)\int_{\frac{h}{2}}^h|u_{x_i}|^2\ud x_1+4C_1^2\int_{\frac{h}{2}}^{\varphi(x')}|\delta(x)u_{x_ix_1}|^2\ud x_1\Big)\\
&=(4m-2)\int_{\frac{h}{2}}^h|\nabla u|^2\ud x_1+4C_1^2\int_{\frac{h}{2}}^{\varphi(x')}|\delta(x)\nabla u_{x_1}|^2\ud x_1,
\end{align*}
thus
\begin{align*}
\int_{0}^{\varphi(x')}|\nabla u|^2\ud x_1&\leq (4m-1)\int_{0}^h|\nabla u|^2\ud x_1+4C_1^2\int_{\frac{h}{2}}^{\varphi(x')}|\delta(x)\nabla u_{x_1}|^2\ud x_1\\
&\leq (4m-1)\int_{0}^h|\nabla u|^2\ud x_1+4C_1^2\int_{0}^{\varphi(x')}|\delta(x)\nabla u_{x_1}|^2\ud x_1.
\end{align*}
Integrating in $x'$ over $\omega$ we discover
\begin{equation}
\label{3}
\int_{\Omega}|\nabla u|^2\ud x\leq (4m-1)\int_{\Omega_h}|\nabla u|^2\ud x+4C_1^2\int_{\Omega}|\delta(x)\nabla u_{x_1}|^2\ud x,
\end{equation}

where $\Omega_h=[0,h]\times\omega.$
In the next step we apply Lemma~\ref{lemma,elliptic.gradient} to the function $u_{x_1}$ in the domain $\Omega$.
It is clear that $u_{x_1}=0$ on $\{x\in\partial\Omega\ : \ x'\in\partial\omega\}$ and that $L(u_{x_1})=L(u)_{x_1}=0$ in $\Omega,$ thus

\begin{align*}
\int_{\Omega}|\delta(x)\nabla u_{x_1}|^2\ud x&\leq \Big(\frac{4n\Lambda^2}{\lambda^2}+1\Big)\int_{\Omega}|u_{x_1}|^2\ud x+\frac{1}{\lambda^2}\int_{\Omega}|\delta(x)^2L(u_{x_1})|^2\\
&=\Big(\frac{4n\Lambda^2}{\lambda^2}+1\Big)\int_{\Omega}|u_{x_1}|^2\ud x.
\end{align*}
Now (\ref{3}) and the last inequality together imply,
$$\int_{\Omega}|\nabla u|^2\ud x\leq (4m-1)\int_{\Omega_h}|\nabla u|^2\ud x+4C_1^2\Big(\frac{4n\Lambda^2}{\lambda^2}+1\Big)\int_{\Omega}|u_{x_1}|^2\ud x.$$

The last inequality coupled with Lemma~\ref{lemma1.ellipcit.korn} applied to the function $u$ in the domain $\Omega_h$ completes the proof.\\

\textbf{Proof of Theorem~\ref{theorem3.ellipcit.hyperplane}}. We make a change of variables $y_1=x_1-a_1-\sum_{i=2}^na_ix_i,$ $y_i=x_i,$ $i\geq 2$ and consider the function $v(y)=u(x)$ to reduce the problem to the case when the left boundary of $\Omega$ is a subspace of the coordinate hyperplane $x_1=0,$ i.e., when Theorem~\ref{theorem2.ellipcit.korn} is applicable. It is easy to check that then the new function $v$ will be defined in the domain $\Omega'=\{y\in\mathbb R^n \ : \ y'\in\omega, \ 0< y_1< \phi(y')\},$ where
$\phi(y')=\varphi(x')-a_1-\sum_{i=2}^na_ix_i.$ Moreover, function $v$ will satisfy the identity
\begin{equation}
\label{L'.elliptic}
L'(v)=(b_1+\sum_{i=2}^nb_ia_i^2)v_{11}-2\sum_{i=2}^nb_ia_iv_{1i}+\sum_{i=2}^nb_iv_{ii}=0.
\end{equation}
Let us prove the following assertion.\\
\textbf{Claim.} The operator $L'$ satisfies conditions (\ref{ellipticity}) and (\ref{bounded.coeff}) with $\Lambda'=\Lambda\left(1+(n-1)(A+A^2)\right)$ and
 $\lambda'=\frac{\lambda_{A}}{n-1}$\\
\textbf{Proof of claim.} Indeed, denoting $L'(v)=\sum_{i,j=1}^n a_{ij}'v_{ij}$ we have for the coefficients
$$a_{11}'=b_1+\sum_{i=2}^nb_ia_i^2,\quad a_{1j}'=a_{j1}'=-b_1a_j,\quad\text{and}\quad a_{jj}'=b_j\quad\text{for}\quad j=2,\dots,n.$$
All coefficients $a_{ij}'$ that do not appear in the above formulas, vanish.
For any $x\in\mathbb R^n$ and any fixed $2\leq j\leq n$ we can estimate like in the proof of Lemma~\ref{lemma,elliptic.gradient2},
\begin{align*}
\sum_{i,j=1}^n a_{ij}'x_ix_j&=b_1x_1^2+\sum_{i=2}^nb_i(a_ix_1-x_i)^2\\
&\geq \lambda x_1^2+\lambda(a_jx_1-x_j)^2\\
&=\lambda\left((1+a_j^2)x_1^2-2a_jx_1x_j+x_j^2\right)\\
&\geq \lambda_{a_j}(x_1^2+x_j^2)\\
&\geq \lambda_{A}(x_1^2+x_j^2),
\end{align*}
thus summing up in $j$ we discover
\begin{align*}
(n-1)\sum_{i,j=1}^n a_{ij}'x_ix_j&\geq \lambda_{A}((n-1)x_1^2+\sum_{j=2}^nx_j^2)\\
&\geq \lambda_{A}|x|^2.
\end{align*}
Nest we derive an upper bound on the sum $\sum_{i=1}^n |a_{ij}'|$ for each $j$. For $j=1$ we have,
\begin{align*}
\sum_{i=1}^n |a_{i1}'|&=b_1+\sum_{i=2}^n(b_ia_i^2+b_i|a_i|)\\
&\leq \Lambda\left(1+(n-1)(A+A^2)\right),
\end{align*}
and for any $2\leq j\leq n$ we have as well,
\begin{align*}
\sum_{i=1}^n |a_{ij}'|&=b_i+b_i|a_i|\\
&\leq \Lambda(1+A),
\end{align*}
as claimed.\\
It is easily seen from the definition of the function $\phi,$ that it is a Lipschitz function with a Lipschitz constant $L'=L+(n-1)A.$
The Jacobian of the change of variables is exactly $1,$ thus it remains to apply Theorem~\ref{theorem2.ellipcit.korn} to the pair $(v,\Omega').$

\section{Korn and Korn-like inequalities in the space dimension two}

\subsection{Fixed boundary conditions}
\label{section.fixed.boundary.conditions}
We start with a definition.
\begin{Definition}
\label{def:l-periodicity}
Let $l,h>0$ and $\Omega=(0,h)\times(0,l).$ The function $u\in C^3(\bar\Omega,\mathbb R)(W^{1,2}(\bar\Omega,\mathbb R))$ is said to be $l$-periodic in $y,$ if there exists a function $v\in C^3([0,h]\times\mathbb R,\mathbb R)(W^{1,2}([0,h]\times\mathbb R,\mathbb R))$ that is $l$-periodic in $y$ and $u(x,y)=v(x,y)$ for all $(x,y)\in \bar\Omega,$ i.e., if it can be extended to an $l$-periodic function in $y$ to the whole $\mathbb R$ preserving the maximal regularity.
\end{Definition}

\begin{Lemma}
\label{lemma1.elliptic.dim2}
 Let $l,h>0$ and $\Omega=(0,h)\times(0,l).$ Assume $a(y)\in C^1[0,l]$ and consider the elliptic operator $L_a(u)=(1+a^2(y))u_{xx}-2a(y)u_{xy}+u_{yy}-a'(y)u_x$ like in Lemma~\ref{lemma,elliptic.gradient2}. Then there exists a constant $C$ depending on $M=\|a\|_{L^\infty(\Omega)}$ and $M_1=\|a'\|_{L^\infty(\Omega)}$ such that if a $u\in C^3(\bar\Omega)$ solution of $L_a(u)=0$ in $\Omega$ satisfies one of the conditions
 \begin{itemize}
 \item[(i)] $u(x,0)=u(x,l)=0$ for all $x\in [0,h],$
 \item[(ii)] The function $u(x,y)$ is $l$-periodic in $y$ and $a(0)=a(l),$
\end{itemize}

 then there holds
$$\|\nabla u\|_{L^2(\Omega)}^2\leq \frac{C(M,M_1)}{h}\|u\|_{L^2(\Omega)}\|u_{x}\|_{L^2(\Omega)}+C(M,M_1)\|u_{x}\|_{L^2(\Omega)}^2.$$
\end{Lemma}

\begin{proof}

First we derive an estimate like (\ref{2}) for the operator $L_a.$  Following the notation in the proof of Lemma~\ref{lemma1.ellipcit.korn} we have by integration by parts,
\begin{align*}
\lambda_a\int_{\Omega_t}|\nabla u|^2&\leq \int_{\Omega_t}(1+a^2(y))u_x^2-2a(y)u_xu_y+u_y^2\\
&=-\int_{\Omega_t}uL_a(u)-\int_{\Omega_t}a'(y)uu_x\\
&+\int_{\{h/2+t\}\times[0,l]}(1+a^2(y))uu_x\ud y-\int_{\{h/2-t\}\times[0,l]}(1+a^2(y))uu_x\ud y,
\end{align*}
as it is easy to see that under either of the conditions $(i)$ or $(ii)$ the boundary integral on the horizontal boundary of $\Omega_t$ vanishes.
Integrating now the inequality over $[0,h/2]$ and applying the Schwartz inequality we obtain
\begin{align*}
\lambda_a\int_0^{\frac{h}{2}}\int_{\Omega_t}&|\nabla u|^2\ud t\leq \int_0^{\frac{h}{2}}\left(\int_{\Omega_t}|a'(y)uu_x|\right)\ud t\\
&+\int_0^{\frac{h}{2}}\left(\int_{\{h/2+t\}\times[0,l]}(1+a^2(y))|uu_x|\ud y+\int_{\{h/2-t\}\times[0,l]}(1+a^2(y))|uu_x|\ud y\right)\ud t\\
&\leq M_1\int_0^{\frac{h}{2}}\left(\int_{\Omega}|uu_x|\right)\ud t+(1+M^2)\int_{\Omega}|uu_x|\\
&\leq \left(\frac{hM_1}{2}+1+M^2\right)\|u\|_{L^2(\Omega)}\|u_x\|_{L^2(\Omega)},
\end{align*}
thus we get like in the proof of Lemma~\ref{lemma1.ellipcit.korn}, that
$$\int_{\Omega_{h/4}}|\nabla u|^2\leq \left(2M_1+\frac{4(1+M^2)}{h}\right)\|u\|_{L^2(\Omega)}\|u_x\|_{L^2(\Omega)},$$
as wished. The rest of the proof is slightly different from the proof of Lemma~\ref{lemma1.ellipcit.korn}. We will apply Lemma~\ref{lemma,elliptic.gradient2} instead of Lemma~\ref{lemma,elliptic.gradient}. The only thing we have to check is that Lemma~\ref{lemma,elliptic.gradient2} is indeed applicable to the function $u_x$ under each of the conditions $(i)$ and $(ii).$ If the condition $(i)$ is satisfied then we have that $u_x(x,0)=u_x(x,l)=0$ for all $x\in(0,h),$ thus condition (\ref{Zero.boundary.integral}) is satisfied. We have furthermore
 $L_a(u_x)=(L_a(u))_x=0$ and therefore Lemma~\ref{lemma,elliptic.gradient2} applies. If now condition $(ii)$ is satisfied then we again have $L_a(u_x)=(L_a(u))_x=0$ and the integrand in (\ref{Zero.boundary.integral}) for $u_x,$ instead of $u,$ takes the same values at points $(x,0)$ and $(x,l)$ and thus the integral in
(\ref{Zero.boundary.integral}) vanishes, hance Lemma~\ref{lemma,elliptic.gradient2} applies, so we achieve the proof.
\end{proof}

Next we prove the following $2D$ version of Theorem~\ref{theorem2.ellipcit.korn}, where the operator $L$ is replaced by $L_a.$
\begin{Lemma}
\label{lemma2.elliptic.dim2}
 Assume $l>0.$ Let $a(y)\in C^1[0,l]$ and let $\varphi\colon[0,l]\to\mathbb R$ be Lipschitz with $H=\sup_{y\in[0,l]}\varphi(y)$ and $h=\inf_{y\in[0,l]}\varphi(y)>0.$ Denote $\Omega=\{(x,y)\in\mathbb R^2 \ : \ y\in[0,l], \ 0<x<\varphi(y)\}$ and
  consider the elliptic operator $L_a(u)=(1+a^2(y))u_{xx}-2a(y)u_{xy}+u_{yy}-a'(y)u_x$ like in Lemma~\ref{lemma,elliptic.gradient2}. Then there exists a constant $C$ depending on $m=H/h,$ $M=\|a\|_{L^\infty(\Omega)}$ and $M_1=\|a'\|_{L^\infty(\Omega)}$ and $L=\mathrm{Lip(\varphi)},$ such that any $u\in C^3(\bar\Omega)$ solution of $L_a(u)=0$ in $\Omega$ satisfying the boundary conditions $u(x)=0$ on the portion $\Gamma=\{(x,y)\in\partial\Omega \ : \ y=0\ \text{or}\ y=l\}$ of the boundary of $\Omega$
fulfills the inequality
$$\|\nabla u\|_{L^2(\Omega)}^2\leq \frac{C(m,M,M_1,L)}{h}\|u\|_{L^2(\Omega)}\|u_{x}\|_{L^2(\Omega)}+C(m,M,M_1,L)\|u_{x}\|_{L^2(\Omega)}^2.$$
\end{Lemma}

\begin{proof}
The proof relies on Lemmas \ref{lemma,elliptic.gradient2}, \ref{lemma,a,epsilon} and \ref{lemma1.elliptic.dim2}. It is actually identical to the proof of Theorem~\ref{theorem2.ellipcit.korn}, that was based on Lemmas $\ref{lemma,elliptic.gradient},$ $\ref{lemma,a,epsilon}$ and $\ref{lemma1.ellipcit.korn},$ therefore we will not repeat it.
\end{proof}

\textbf{Proof of Theorem~\ref{theorem1.second.korn.dim2}.} By a density argument one can without loss of generality assume that $U$ is of class $C^\infty$ up to the boundary of $\Omega.$ Moreover, one can without loss of generality assume that the displacement $U$ is harmonic in $\Omega$, see [\ref{bib.Oleinik.shamaev.Yosifian}, page 18] for a proof. We make a change of variables $x_1=x-\varphi_1(y)$ and $y_1=y,$ which transforms the domain $\Omega$ into $\Omega^1$ in the coordinate plane $OX_1Y_1.$ Consider the new displacement $U^1=(u^1,v^1),$ where $u^1(x_1,y_1)=u(x,y)$ and $v^1(x_1,y_1)=v(x,y).$ It is easy to verify that
\begin{equation}
\label{L(u1)=0}
L_{\varphi_1'}(u^1)=\triangle u=0 \qquad \text{in}\qquad \Omega^1.
\end{equation}
We have furthermore that
$$
\nabla U=
\begin{bmatrix}
u_{x_1}^1 & -\varphi_1'(y)u_{x_1}^1+u_{y_1}^1\\
v_{x_1}^1 & -\varphi_1'(y)v_{x_1}^1+v_{y_1}^1
\end{bmatrix},
$$
thus we have to estimate $\|-\varphi_1'(y)u_{x_1}^1+u_{y_1}^1\|_{L^2(\Omega^1)}$ in terms of $\|e(U)\|_{L^2(\Omega)}$ and $\|u^1\|_{L^2(\Omega^1)}$ as the Jacobian of the change of variables is exactly $1.$ By (\ref{L(u1)=0}) and the fact that $\varphi_2-\varphi_1$ determines the right boundary of $\Omega_1$ and has a Lipshitz constant less than $\rho_1+\rho_2$ we can owe to Lemma~\ref{lemma2.elliptic.dim2} and the triangle inequality to establish
\begin{align*}
\|-\varphi_1'&(y)u_{x_1}^1+u_{y_1}^1\|_{L^2(\Omega^1)}^2\leq 2(\|u_{y_1}^1\|_{L^2(\Omega^1)}^2+\|\varphi_1'(y)u_{x_1}^1\|_{L^2(\Omega^1)}^2)\\
&\leq \frac{C(m,\rho_1,\rho_1',\rho_2)}{h}\|u^1\|_{L^2(\Omega^1)}\|u_{x_1}^1\|_{L^2(\Omega^1)}+C(m,\rho_1,\rho_1',\rho_2)\|u_{x_1}^1\|_{L^2(\Omega^1)}^2+\rho_1^2\|u_{x_1}^1\|_{L^2(\Omega^1)}^2)\\
&\leq \frac{C(m,\rho_1,\rho_1',\rho_2)+\rho_1^2}{h}\|u^1\|_{L^2(\Omega^1)}\|u_{x_1}^1\|_{L^2(\Omega^1)}+(C(m,\rho_1,\rho_1',\rho_2)+\rho_1^2)\|u_{x_1}^1\|_{L^2(\Omega^1)}^2\\
&= \frac{C(m,\rho_1,\rho_1',\rho_2)+\rho_1^2}{h}\|u\|_{L^2(\Omega)}\|u_{x}\|_{L^2(\Omega)}+(C(m,\rho_1,\rho_1',\rho_2)+\rho_1^2)\|u_{x}\|_{L^2(\Omega)}^2\\
&\leq\frac{C(m,\rho_1,\rho_1',\rho_2)+\rho_1^2}{h}\|u\|_{L^2(\Omega)}\|e(U)\|_{L^2(\Omega)}+(C(m,\rho_1,\rho_1',\rho_2)+\rho_1^2)\|e(U)\|_{L^2(\Omega)}^2.
\end{align*}
The other component $v_{x_1}^1$ of the gradient is estimated by $-\varphi_1'(y)u_{x_1}^1+u_{y_1}^1$ and $e(U)$ via triangle inequality.

\subsection{Periodic boundary conditions}
\label{section.periodic.boundary.conditions}

\begin{Lemma}
\label{lemma.periodic.1.dim2}
Let $l,$ $a(y),$ $M,$ $M_1$, $\varphi,$ $h,$ $H,$ $L$ and $\Omega$ be as in Lemma~\ref{lemma2.elliptic.dim2}. Assume furthermore that $a(0)=a(l).$ Then there exists a constant $C$ depending on $m=H/h,$ $M,$ $M_1$ and $L$ such that if a $u\in C^3(\bar\Omega)$ solution of $L_a(u)=0$ is $l$-periodic in $y,$ then
$$\|u_y\|_{L^2(\Omega)}^2\leq \frac{C(m,M,M_1,L)}{h}\|u\|_{L^2(\Omega)}\|u_x\|_{L^2(\Omega)}+C(m,M,M_1,L)\|u_x\|_{L^2(\Omega)}^2.$$
\end{Lemma}

\begin{proof}
 The proof is analogues to the proof of Lemma~\ref{lemma2.elliptic.dim2} and relies again on Lemmas \ref{lemma,elliptic.gradient2}, \ref{lemma,a,epsilon} and \ref{lemma1.elliptic.dim2}.
\end{proof}

\begin{Lemma}
\label{lemma.periodic.2.dim2}
Let $l,$ $\varphi_1,$ $\varphi_2,$ $h,$ $H$ $\rho_1$, $\rho_2,$ $\rho_1'$ and $\Omega$ be as in Theorem~\ref{theorem1.second.korn.dim2}. Assume furthermore that $\varphi_1(0)=\varphi_1(l),$ $\varphi_1'(0)=\varphi_1'(l)$ and $\varphi_2(0)=\varphi_2(l).$ Then there exists a constant $C$ depending on $m=H/h,$ $\rho_1,$ $\rho_2$ and $\rho_1'$ such that if the function $u\in C^3(\bar\Omega)$ is $l$-periodic in $y$ and is harmonic in $\Omega,$ then
$$\|u_y\|_{L^2(\Omega)}^2\leq \frac{C(m,\rho_1,\rho_2,\rho_1')}{h}\|u\|_{L^2(\Omega)}\|u_x\|_{L^2(\Omega)}+C(m,\rho_1,\rho_2,\rho_1')\|u_x\|_{L^2(\Omega)}^2.$$
\end{Lemma}

\begin{proof}
The proof can be carried out as done in the proof of Theorem~\ref{theorem1.second.korn.dim2}, where $\varphi_1'(y)$ plays the role of $a(y).$ It is based on Lemma~\ref{lemma.periodic.1.dim2} and the change of variable argument in the proof of Theorem~\ref{theorem1.second.korn.dim2}. Notice that in the proof of Theorem~\ref{theorem1.second.korn.dim2} we actually estimated the norm of the second derivative $u_y$ of $u$ in terms of the norms of the first derivative $u_x$ and $u$ itself.
\end{proof}

\textbf{Proof of Theorem~\ref{theorem.dim2.pariodic}}. It is a direct consequence of Lemma~\ref{lemma.periodic.2.dim2} and the fact that one can without loss of generality assume that $u$ is harmonic.
\vspace{0.3cm}

 \textbf{\large{Acknowledgement}}\\

The present results have been obtained while the author was a postdoctoral fellow at Temple University. The author is very grateful to Y. Grabovsky for supporting his stay at Temple University and for many helpful discussions. He is also thankful to the anonymous reviewer for very valuable comments. The
material is based upon work supported by the National Science Foundation Grant No. 1008092 (Y.G.).

\end{document}